\documentclass[a4paper,11pt,reqno]{amsart}
\usepackage{amsthm}
\usepackage{amsmath}
\usepackage{enumerate}
\usepackage{amsfonts}
\usepackage{amssymb}
\usepackage{fullpage}
\usepackage{amsmath,amscd}
\usepackage{stmaryrd}
\usepackage{graphicx}
\usepackage{array}
\usepackage{amsmath,amssymb,amsfonts,dsfont}
\usepackage[utf8]{inputenc} 
\usepackage[english]{babel}
\usepackage[T1]{fontenc} 
\usepackage{graphicx}
\usepackage{float}
\usepackage{pdfpages}
\usepackage[a4paper, margin = 3cm, bottom = 3cm]{geometry}
\usepackage{ifpdf}
\usepackage{marginnote}
\usepackage{hyperref}	
\hypersetup{pdfborder=0 0 0, 
	    colorlinks=true,
	    citecolor=black,
	    linkcolor=blue,
	    urlcolor=red,
	    pdfauthor={Guillaume Tahar}
	   }
\newtheorem{thm}{Theorem}[section]
\newtheorem{cor}[thm]{Corollary}
\newtheorem{prop}[thm]{Proposition}
\newtheorem{lem}[thm]{Lemma}
\theoremstyle{definition}
\newtheorem{defn}[thm]{Definition}
\theoremstyle{remark}

\theoremstyle{definition}

\theoremstyle{definition}

\theoremstyle{definition}

\numberwithin{equation}{section} 
\title{Chamber structure of modular curves $X_{1}(N)$}
\author{Guillaume Tahar} 
\address[Guillaume Tahar]{Institut de Math{\'e}matiques de Jussieu - UMR CNRS 7586}
\email{guillaume.tahar@imj-prg.fr}
\date{February 5, 2019}
\begin{document}
\begin{abstract}
Modular curves $X_{1}(N)$ parametrize elliptic curves with a point of order $N$. They can be identified with connected components of projectivized strata $\mathbb{P}\mathcal{H}(a,-a)$ of meromorphic differentials. As strata of meromorphic differentials, they have a canonical walls-and-chambers structure defined by the topological changes in the flat structure defined by the meromorphic differentials.
We provide formulas for the number of chambers and an effective means for drawing the incidence graph of the chamber structure of any modular curve $X_{1}(N)$. This defines a family of graphs with specific combinatorial properties. This approach provides a geometrico-combinatorial computation of the genus and the number of punctures of modular curves $X_{1}(N)$. Although the dimension of a stratum of meromorphic differentials depends only on the genus and the numbers of the singularities, the topological complexity of the stratum crucially depends on the order of the singularities.\newline
\end{abstract}
\maketitle
\setcounter{tocdepth}{1}
\tableofcontents

\section{Introduction}

Very little is known about the topology of strata of holomorphic differentials on compact Riemann surfaces. The conjecture of Kontsevich and Zorich that strata are $K(\pi,1)$-spaces is wide open, see \cite{KZ}. There are some partial results for hyperelliptic components or low genus strata, see \cite{LM}. Contrary to the situation of holomorphic differentials where there are a finite number of strata of given genus, there are infinitely many strata of meromorphic differentials. In genus zero, projectivized strata are configuration spaces of points on the sphere so the topology is well understood. In this paper, we focus on the first interesting case, the family of strata $\mathcal{H}(a,-a)$ where $a\geq2$. Such strata are moduli spaces of pairs $(X,\phi)$ up to biholomorphism where $X$ is a Riemann surface of genus one and $\phi$ is a meromorphic $1$-form with exactly one zero of order $a$ and a pole of order $a$. Assuming genus is nonzero, strata $\mathcal{H}(a,-a)$ are the only strata of meromorphic differentials whose complex dimension is two. Therefore, projectivized strata $\mathbb{P}\mathcal{H}(a,-a)$ are complex curves whose genus and number of punctures can be computed. Just like the modular curve $X_{1}(1)$ is identified with the space of flat tori with a marked point (the case $a=0$), modular curves $X_{1}(a)$ are identified with connected components of projectivized strata $\mathbb{P}\mathcal{H}(a,-a)$ for $a\geq2$. Projectivized strata are the quotients of the strata by the action of $\mathbb{C}^{*}$ \newline

Since strata of meromorphic differentials have a natural walls-and-chambers structure (see \cite{Ta}), their topology may be grasped through a topologico-combinatorial study. Strata $\mathcal{H}(a,-a)$ are simple enough to allow complete determination of their topological features. For every value of $a$, we can draw the incidence graph of the chambers. This provides the first step of a systematic cartography of strata of meromorphic differentials. Incidentally, we get a new way to compute the genus and the number of punctures of modular curves $X_{1}(a)$ that does not make use of number theory or algebraic geometry.\newline
In our context, chambers of projectivized strata are topological disks (with, in some cases, a puncture inside) and walls are real codimension one submanifolds which are affine in the period coordinates. In other words, walls are straight lines. They meet each other at the punctures of the algebraic complex curve.\newline

Many quantities associated to strata only depends on the genus and the number of singularities rather than on their orders. For example, the complex dimension of a stratum $\mathcal{H}(a_1,\dots,a_n,-b_1,\dots,-b_p)$ of meromorphic differentials with zeroes of order $a_1,\dots,a_n$ and poles of order $b_1,\dots,b_p$ on a surface of genus $g$ is $2g+n+p-2$, see \cite{BS}. The same quantity is the lower bound on the number of saddle connections for surfaces of this stratum, see \cite{Ta}. On the contrary, in this paper, we show that the genus of $\mathbb{P}\mathcal{H}(a,-a)$ grows quadratically in $a$. Therefore, the topological complexity takes into account the order of the singularities.\newline
It should be noted that the dimension of a stratum of meromorphic differentials is $2g+n+p-2$ whereas the dimension of a stratum of holomorphic differentials is $2g+n-1$. Indeed, the sum of the residues at the poles is zero so there is automatically an additional relations between the periods of the differential.

\section{Statement of main results}

Connected components of strata of meromorphic differentials were classified by Boissy in \cite{Bo}. Two flat surfaces of $\mathcal{H}(a,-a)$ belong to the same connected component if and only if they have the same rotation number. The rotation number is defined in terms of indices of loops in the surface.

\begin{defn}
Let $\gamma$ be a simple closed curve in a translation surface with (or without) poles. $\gamma$ is parametrized by arc-length and passes only through regular points. We set $\gamma'(t)=e^{i\theta(t)}$.\newline
We have $\dfrac{1}{2\pi}\int_{0}^{T} \theta'(t)dt \in \mathbb{Z}$. This number is the topological index $ind(\gamma)$ of $\gamma$.
\end{defn}

In particular, the topological index of a simple closed curve around a singularity of order $k$ is $1+k$.\newline

The rotation number of a flat surface $X$ of $\mathcal{H}(a,-a)$ is $rot(X)=gcd(a,ind(\beta),ind(\gamma))$ where $\beta$ and $\gamma$ are simple loops forming an homology basis ($X$ is a torus). The rotation number is a topological invariant of connected components of strata. Boissy proved in Theorem 1.1 of \cite{Bo} that the rotation number completely classifies connected components of $\mathcal{H}(a,-a)$. There is a connected component $\mathcal{C}_{d}^{a}$ for every integer divisor $d$ of $a$ except $a$ itself. We denote by \textit{principal component} of the stratum the component where the rotation number of the flat surfaces is $1$.\newline

In the following theorem, we prove that the isomorphism type of connected components depends only on the ratio between the order of the singularities and the rotation number of the component. Therefore, we will be allowed to focus on the principal connected component of each stratum, that is the connected component where the rotation number is equal to one.\newline

For any $N \geq 2$, modular curves $X_{1}(N)$ are compactifications of the quotient of the upper half-plane $\mathbb{H}$ by the congruence groups $\Gamma_{1}(N)=\left \{ \begin{pmatrix} a & b \\ c & d\end{pmatrix}  \in SL_{2}(\mathbb{Z}) : \ a\equiv d\equiv 1\mod N, c\equiv 0 \mod N \right \}$, see \cite{DS} for more details.\newline

\begin{thm} For any $a \geq 2$, every connected component $\mathcal{C}_{d}^{a}$ of $\mathbb{P}\mathcal{H}(a,-a)$ is biholomorphic to the modular curve $X_{1}\left( \dfrac{a}{d}\right) $.\newline
\end{thm}

This theorem provides a conceptual reason for the nonexistence of a component whose rotation number is equal to $a$. Indeed, this component would have been isomorphic to $X_{1}(1)$ which is empty.\newline

It should also be noted that the isomorphism between connected components does not preserve the action of $GL^{+}(2,\mathbb{R})$ since the number of open orbits may vary.\newline

Theorem 2.2 is proved in section 3.\newline

We already know that the topology of connected components is given by modular curves. Connected components of different strata may be isomorphic if they correspond to the same modular curve. That is not all. Their walls-and-chambers structure will also be conjugated. This means that the topological pairs formed by the connected components and their discriminant (the union of the walls) are homeomorphic.

\begin{thm} Walls-and-chambers structures of connected components $\mathcal{C}_{d}^{a}$ and $\mathcal{C}_{k.d}^{k.a}$ are conjugated.
\end{thm}

To a certain extent, we can say the walls-and-chambers structure exists directly on the modular curve $X_{1}\left( \dfrac{a}{d}\right)$. For any modular curve, we get a graph of chambers describing the adjacency relations between them. We get a family of graphs $\mathcal{G}_{a}$ where $a\geq2$, see Figure 1 for the first graphs of the family.\newline

\begin{figure}
\includegraphics[scale=0.3]{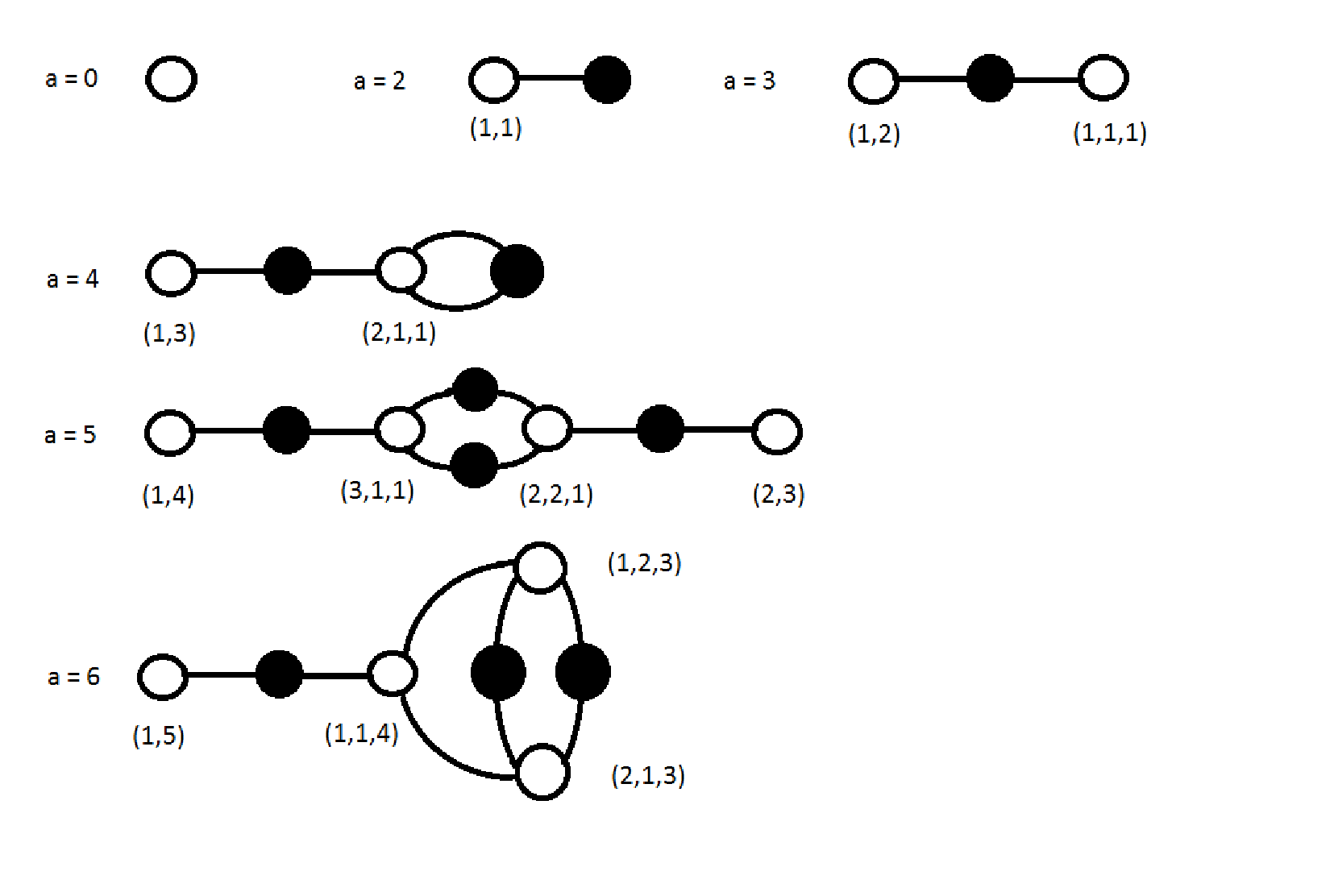}
\caption{The first elements of the family of graphs $\mathcal{G}_{a}$. Black and white vertices correspond to different types of chambers.}
\end{figure}

Theorem 2.3 is proved in section 4.\newline

Theorem 2.4 computes the genus and the number of punctures of modular curves $X_{1}(N)$. This is a well known result of number theory. What is new is that while usual proofs are algebro-geometric (see \cite{DS} for a complete study), our proof involves only flat geometry.\newline
In the following, $\phi$ is Euler's totient function.\newline

\begin{thm}
For any $N\geq5$, $X_{1}(N)$ is a complex curve of genus $g$ with $l$ cusps where we have $g=1+\dfrac{N^{2}}{24}\prod \limits_{p \vert N}(1-\dfrac{1}{p^{2}})-\dfrac{1}{4}\sum \limits_{d \vert N} \phi(d)\phi(N/d)$ and $l=\dfrac{1}{2}\sum \limits_{d \vert N} \phi(d)\phi(N/d)$.\newline
There are three exceptional cases: $X_{1}(4)$ is a complex curve of genus zero with three cusps, $X_{1}(3)$ is a complex curve of genus zero with two cusps and an orbifold point of order three whereas $X_{1}(2)$ is a complex curve of genus zero with two cusps and an orbifold point of order two.\newline
\end{thm}

These formula simplify if we consider prime numbers.

\begin{cor}
For a given prime number $p$, we have $X_{1}(p) \simeq \mathbb{P}\mathcal{H}(p,-p)$. As a complex algebraic curve, this modular curve is of genus $\dfrac{(p-5)(p-7)}{24}$ with $p-1$ punctures.
\end{cor}

Finally, we also confirm that the modular curves $X_{1}(N)$ (and the connected components of strata that are isomorphic to them) are topological spheres with punctures when $2 \leq N \leq 10$ and $N=12$.\newline

Theorem 2.4 and Corollary 2.5 are proved in section 5.\newline

The structure of the paper is the following:\newline
- In Section 3, we recall the background and tools useful to study translation surfaces with poles: flat metric, saddle connections, the moduli space, the core of a surface and its associated wall-and-chambers structure, the action of $GL^{+}(2,\mathbb{R})$, description of connected components in terms of divisors.\newline
- In Section 4, we prove general geometric results about chambers and combinatorial relations between them, drawing their graph of chambers. We also describe projectivized strata $\mathbb{P}\mathcal{H}(2,-2)$, $\mathbb{P}\mathcal{H}(3,-3)$ and $\mathbb{P}\mathcal{H}(4,-4)$.\newline
- In Section 5, we establish formula to count chambers, connected components of the discriminant and to compute the topological invariants of the connected components of strata.\newline 

\textit{Acknowledgements.} I thank Quentin Gendron, Carlos Matheus and Anton Zorich for many interesting discussions.\newline

\section{Definitions and tools}

\subsection{Flat structures}

Let $X$ be an elliptic curve and let $\phi$ be a meromorphic $1$-form with a unique zero of order $a$ and a unique pole of order $a$ with $a\geq2$. Outside these two singularities, integration of $\phi$ gives local coordinates whose transition maps are of the type $z \mapsto z+c$. The pair $(X,\phi)$ seen as a smooth surface with such a translation atlas is called a \textit{translation surface with poles}.\newline

In a neighborhood of the zero of the differential, the metric induced by $\phi$ admits a conical singularity of angle $\left(2a+2 \right)\pi$. The neighborhood of the pole of order $a$ is obtained by taking an infinite cone of angle $(2a-2)\pi$ and removing a compact neighborhood of the conical singularity.

\subsection{Moduli space}

If $(X,\phi)$ and $(X',\phi')$ are flat surfaces such that there is a biholomorphism $f$ from $X$ to $X'$ such that $\phi$ is the pullback of $\phi'$, then $f$ is an isometry for the flat metrics defined by $\phi$ and $\phi'$.\newline

As in the case of Abelian differentials, we define the moduli space of meromorphic differentials as the space of equivalence classes of flat surfaces with poles of higher order $(X,\phi)$ up to biholomorphism preserving the differential.\newline

We denote by $\mathcal{H}(a,-a)$ the \textit{stratum} of meromorphic differentials with a unique zero and a unique pole that have both the same degree $a$. When we consider differentials up to multiplication by a nonzero complex number, we get the projectivized stratum $\mathbb{P}\mathcal{H}(a,-a)$.

On each stratum, $GL^{+}(2,\mathbb{R})$ acts by composition with coordinate functions, see \cite{Zo}. Since neighborhoods of poles have infinite area, we cannot normalize the area of the surface and thus must consider the full action of $GL^{+}(2,\mathbb{R})$.

\subsection{Saddle connections}

\begin{defn} A \textit{saddle connection} is a geodesic segment joining two conical singularities of the flat surface such that all interior points are not conical singularities.
\end{defn}

Every saddle connection represents a class in the relative homology group $H_{1}(X\setminus \lbrace P \rbrace, \lbrace N \rbrace)$ where $P$ and $N$ are respectively the pole and the zero of the meromorphic $1$-form. Strata are complex-analytic orbifolds with local coordinates given by the period map of maximal system of saddle connections, see \cite{Zo}.

\subsection{Connected components}

In addition to a classification of connected components, elementary algebraic geometry allows a complete description of the topology of these components. Components of projectivized strata are modular curves.

\begin{proof}[Proof of Theorem 2.2]
We consider stratum $\mathcal{H}(a,-a)$ with $a\geq2$. An element of $\mathcal{H}(a,-a)$ is given, up to a constant multiple, by a pair $(X,D)$ where $X=\mathbb{C}/\Gamma$ is a complex torus and $D=-aP+aQ$ is a divisor on $X$. We have $\Gamma=\mathbb{Z}u+\mathbb{Z}v$. Without loss of generality, we assume $P=0$. As a consequence of Abel's theorem characterizing divisors coming from meromorphic differentials, we have $aQ \in \Gamma$, see Chapter 3 in \cite{DS}. Therefore, $Q$ is of the form $\dfrac{m}{a}u+\dfrac{n}{a}v$ where $m$ and $n$ are integers. Since we can deform continuously the lattice that defines the torus, different coordinates of $Q$ that are image of each other by basis change operations belong to the same connected component. In other words, we can deform $Q$ into a divisor whose coordinates are $\dfrac{m}{a}$ and $\dfrac{n}{a}v$ in a basis $(u,u+v)$ or $(u+v,v)$. Therefore, divisors of the form $\dfrac{m+kn}{a}u+\dfrac{n+k'm}{a}v$ with $k,k' \in \mathbb{Z}$ correspond to meromorphic differentials in the same connected component. Every basis change is generated by these two operations so there is exactly one connected component for each value $l$ of $gcd(m,n,a)$. Then, $k=\dfrac{a}{l}$ is the smaller integer such that $kQ \in \Gamma$. In other words, $Q$ is a point of $k$-torsion and $l$ may be any integer divisor of $a$ different from $a$ (because $Q \not\in \Gamma$). In the following, $l$ is the algebraic invariant of the component. There is exactly one connected component of $\mathcal{H}(a,-a)$ for each arithmetic divisor $d$ of $a$ different from $a$.\newline
Following \cite{DS}, modular curves $X_{1}(N)$ are the moduli spaces of elliptic curves with a point of $N$-torsion. In other words, $X_{1}(N)$ parametrizes isomorphism classes of pairs $(X,P)$ where $X$ is an elliptic curve and $P$ is a point of $X$ such that $N.P=0$ and $k.P \neq 0$ for any $1 \leq k \leq n-1$. Therefore, projectivized strata $\mathbb{P}\mathcal{H}(a,-a)$ are disjoint unions of modular curves. For any integer divisor $d$ of $a$ different from $a$, there is a unique connected component of meromorphic differentials where the pole is a point of $\dfrac{a}{d}$-torsion. Such component is biholomorphic to $X_{1}\left( \dfrac{a}{d}\right)$.\newline
There is a geometric interpetation of this number: it is the greatest common divisor of the topological degrees of loops in every flat surface of the connected component.  This number is the flat invariant of the connected component. A direct construction show that every integer divisor $d$ of $a$ different from $a$ is realized in the stratum (see \cite{Bo} for systematic constructions of connected components). In the following, we prove that the flat invariant is equal to the algebraic invariant.\newline
Since there is exactly one connected component for each divisor of $a$ different from $a$, we just have to prove that one invariant divides the other to prove that they are equal. We prove that the algebraic invariant divides the flat invariant.\newline
We consider the connected component $\mathcal{C}^{a}_{d}$ of $\mathcal{H}(a,-a)$ whose algebraic invariant is $d$. The pole is a point of $\dfrac{a}{d}$-torsion. Therefore, there exists a meromorphic differential in $\mathcal{C}^{a}_{d}$ that is of the form $f^{d}(z)dz$ where $f(z)dz$ is a meromorphic differential of $\mathcal{C}^{d'}_{1}$ where $d'=\dfrac{a}{d}$. We denote by $b$ the flat invariant of $\mathcal{C}^{d'}_{1}$. We consider a loop $\gamma$ passing only through the regular points of the torus. Since we have $arg(f^{d}(\gamma(t)))=d.arg(f(\gamma(t)))$, the flat invariant of $f^{d}(z)dz$ in $\mathcal{C}^{a}_{d}$ is $bd$. Thus, the algebraic invariant divides the flat invariant in every connected component. Consequently, these two invariants coincide.
\end{proof}

\subsection{Core of a translation surface with poles}

The core of a flat surface of infinite area was introduced in \cite{HKK} and systematically used in \cite{Ta}. It is the convex hull of the conical singularities of a flat surface. Since all saddle connections belong to it, the core encompasses most of the qualitative (see walls-and chambers structure in subsection 3.7) and quantitative (periods of the homology) information about the geometry of the flat surface.

\begin{defn} A subset $E$ of a flat surface is \textit{convex} if and only if every element of any geodesic segment between two points of $E$ belongs to $E$.\newline
The \textit{core} of a flat surface $(X,\phi)$ is the convex hull $core(X)$ of its conical singularities.\newline
$\mathcal{I}\mathcal{C}(X)$ is the interior of $core(X)$ in $X$ and $\partial\mathcal{C}(X) = core(X)\ \backslash\ \mathcal{I}\mathcal{C}(X)$ is its boundary.\newline
The \textit{core} is said to be degenerated when $\mathcal{I}\mathcal{C}(X)=\emptyset$ that is when $core(X)$ is just the graph $\partial\mathcal{C}(X)$.
\end{defn}

\begin{lem} Let $X$ be a flat surface of $\mathcal{H}(a,-a)$, then $X \setminus core(X)$ is a topological disk that contains the unique pole. We refer to this disk as the \textit{domain of the pole}.
\end{lem}

\begin{proof}
Following Proposition 2.3 in \cite{HKK}, $core(X)$ is a deformation retract of $X \backslash \lbrace P \rbrace$ where $P$ is the unique pole.
\end{proof}

\begin{lem} For any translation surface with poles $X$, $\partial\mathcal{C}(X)$ is a finite union of saddle connections.
\end{lem}

\begin{proof} See Proposition 2.2 in \cite{HKK}.
\end{proof}

An ideal triangulation in a topological surface is a maximal collection of pairwise disjoint non-homotopic topological arcs. Connected components of such a decomposition are called \textit{ideal triangles}. They are topological disks bounded by three arcs. Vertices of an ideal triangle may coincide.\newline
There is a formula that relates the maximal number of noncrossing saddle connections of a translation surface with poles with the number of ideal triangles in any ideal triangulation of the core of the surface.

\begin{lem}
Let $|A|$ be the maximal number of noncrossing saddle connections of a flat surface that belongs to $\mathcal{H}(a,-a)$, then we have $|A| = 2+t$ where $t$ is the number of ideal triangles in any ideal triangulation of $core(X)$. Furthermore, $|A| \leq 4$.
\end{lem}

\begin{proof} See Lemma 4.10 in \cite{Ta}.
\end{proof}

\subsection{Discriminant and walls-and-chambers structure}

Strata of translation surfaces with poles decompose into chambers where the qualitative shape of the core is the same. This means that the topological pair $(X,\partial\mathcal{C}(X))$ is the same up to homeomorphism. The \textit{discriminant} is the locus that separates these chambers from each other.

\begin{defn} A translation surface with poles $X$ belongs to the \textit{discriminant} of the stratum if and only if there two consecutive saddle connections in the boundary of the domain of the pole (which is a topological disk, see Lemma 3.4) share an angle of $\pi$. \textit{Chambers} are the connected components of the complementary to the discriminant in the strata.
\end{defn}

The following lemma is proved as Proposition 4.12 in \cite{Ta}.

\begin{lem} The discriminant is a $GL^{+}(2,\mathbb{R})$-invariant hypersurface of real codimension one in the stratum.
\end{lem}

The topological structure on a translation surface with poles $(X,\phi)$ defined by the embedded graph $\partial\mathcal{C}(X)$ is invariant along the chambers. In particular the number of saddle connections of its boundary depend only on the chamber (see Proposition 4.13 in \cite{Ta} for details).\newline

\section{Chambers}

In strata $\mathcal{H}(a,-a)$, types of topological shapes of the core are very restricted.

\begin{prop}
Every chamber of $\mathbb{P}\mathcal{H}(a,-a)$ belongs to one of the following three types:\newline
(i) \textbf{Chambers of "cylinder" type} where $core(X)$ is a cylinder bounded by a pair of saddle connections (see Figure 2).\newline
(ii) \textbf{Chambers of "degenerate" type} where $core(X)$ is a pair of saddle connections (see Figure 3).\newline
(iii) \textbf{Chambers of "triangle" type} where $core(X)$ is an ideal triangle (see Figure 4).
\end{prop}

\begin{figure}
\includegraphics[scale=0.3]{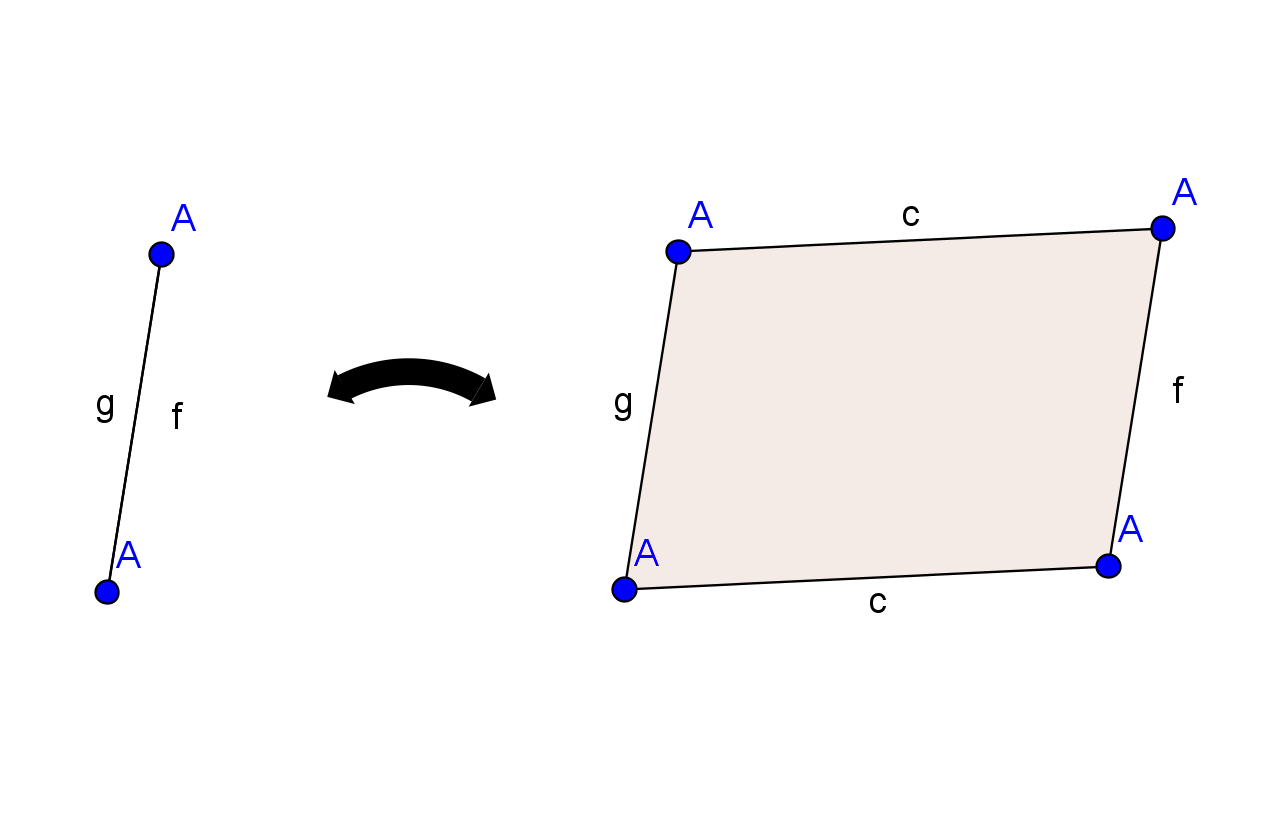}
\caption{A surface of the chamber of "cylinder" type in $\mathcal{H}(2,-2)$.}
\end{figure}

\begin{figure}
\includegraphics[scale=0.3]{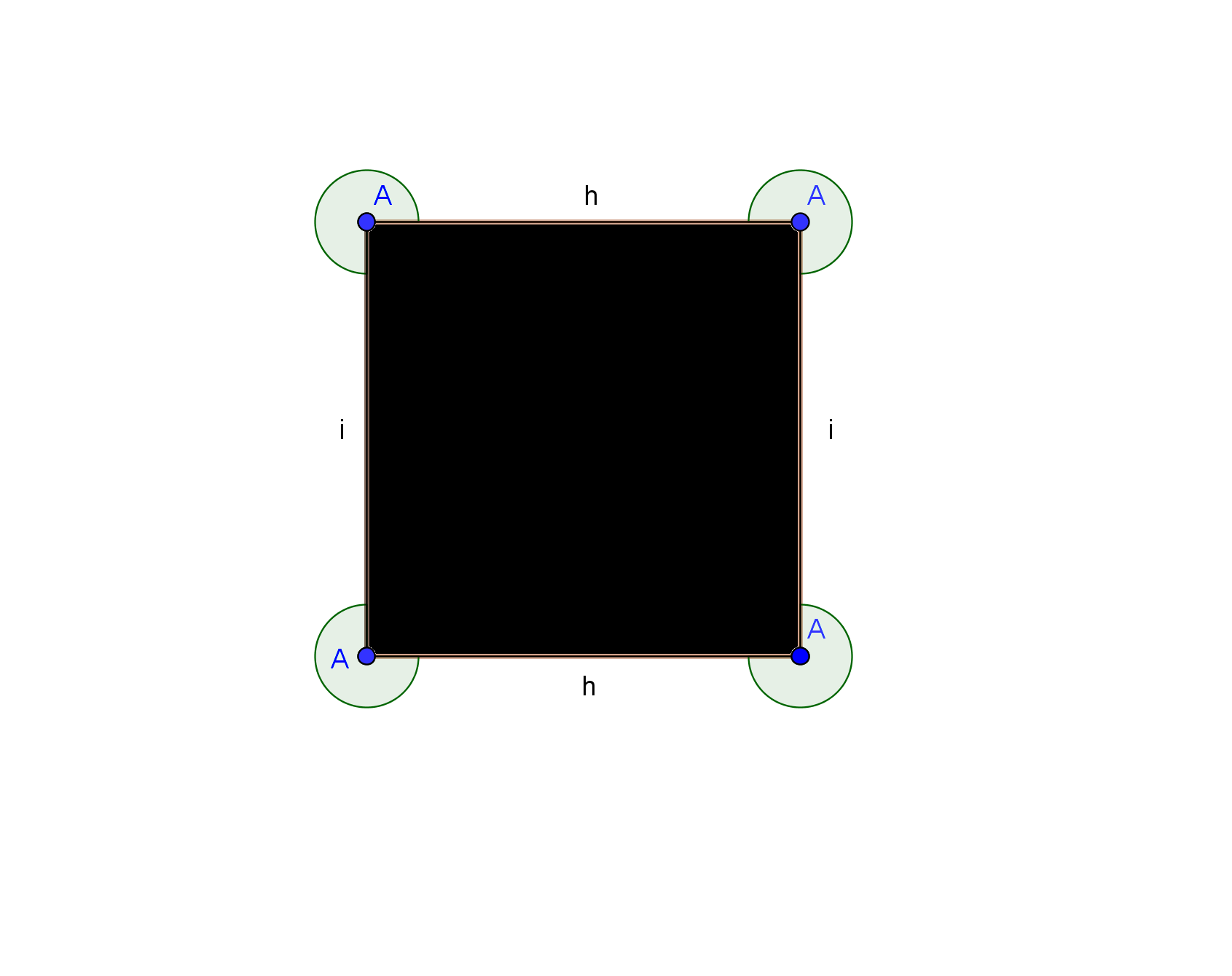}
\caption{A surface of the chamber of "degenerate" type in $\mathcal{H}(2,-2)$.}
\end{figure}

\begin{figure}
\includegraphics[scale=0.3]{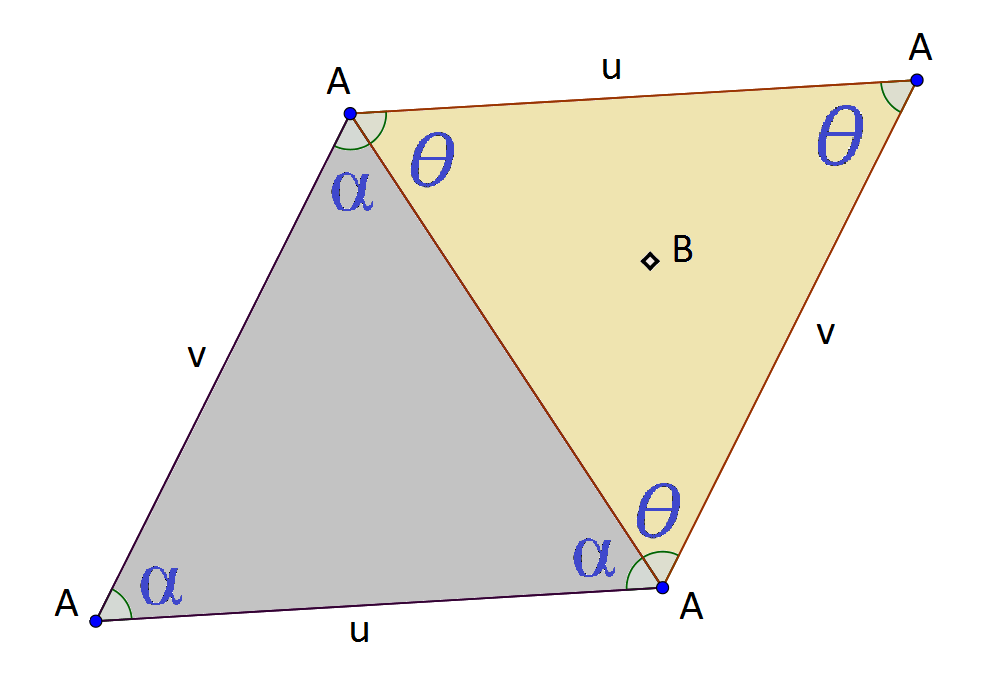}
\caption{A surface of the chamber of "triangle" type in $\mathcal{H}(3,-3)$. The pole of order three is denoted by $B$. Magnitudes of angles are $\alpha = \dfrac{\pi}{3}$ and $\theta = \dfrac{7\pi}{3}$}
\end{figure}

\begin{proof}
Following Lemma 3.5, we already know that for any flat surface in a chamber of $\mathcal{H}(a,-a)$, the maximal number of noncrossing saddle connections is $2+t$ where $t$ is the number of ideal triangles in any triangulation of the core. Besides, we have $0 \leq t \leq 2$. Number $t$ is constant in every chamber. If $t=0$, the chamber is of "degenerate type". If $t=1$, there are exactly three noncrossing saddle connections and they bound the unique ideal triangle. The chamber is of "triangle" type. Finally, if $t=2$, there are four noncrossing saddle connections. Therefore, the two ideal triangles are necessarily connected by exactly two of their sides (otherwise, there would be at least five or six noncrossing saddle connections). Thus, we have a cylinder formed by two ideal triangles and bounded by a pair of parallel saddle connections. The chamber is of "cylinder" type.\newline
\end{proof}

\subsection{Description of the chambers}

Since the walls-and-chambers structure is $GL^{+}(2,\mathbb{R})$-invariant, we will work with projectivized strata $\mathbb{P}\mathcal{H}(a,-a)$ which are the quotients of strata $\mathcal{H}(a,-a)$ by the action of $\mathbb{C}^{*}$. For each type of chamber we will choose a suitable normalization. Building from that, we will be able to provide a parametrization of each chamber and to figure out their shape. Finally, chambers of the same type will be distinguished by a system of combinatorial invariants.

\begin{prop}
In chambers of "cylinder" type of $\mathbb{P}\mathcal{H}(a,-a)$, we normalize the length and the direction of the two saddle connections that bound the core. Flat surfaces are characterized by two invariants:\newline
(i) A bipartition of $a$ as a sum of two nonzero integers. It corresponds to the repartition of a total angle of $2a\pi$ between the two sectors of the domain of the pole. This discrete invariant distinguish the different chambers of the same type in the stratum.\newline
(ii) The shape of the cylinder. The flat surfaces in the chamber are parametrized by $\mathbb{H}/\mathbb{Z}$. Consequently, every such chamber is the image of $GL^{+}(2,\mathbb{R})$ under projection.\newline
In the boundary of such chambers, the cylinder degenerates and the core becomes a pair of saddle connections. Flat surfaces of the boundary are parametrized by the real ratio between their lengths.\newline
These chambers are topological disks with a unique boundary component. They have a cusp in their interior (corresponding to the case of an infinitely thin cylinder) and a cusp in their boundary (if one saddle connection in the core shrinks).\newline
There is a chamber of "cylinder type" for every integer unordered bipartition $a=k+k'$ where $k,k'>0$. The chamber corresponding to the bipartition $a=k+k'$ belongs to the connected component where the rotation number of surfaces is $gcd(k,a)$.
\end{prop}

\begin{proof}
The holonomy vector of the two parallel saddle connections that form the boundary of the core is normalized. The geometry of the flat surface is determined by the shape of the cylinder and the two angles of the domain of the pole. These two angles are non zero integer multiples of $2\pi$ and these multiples appear as the topological indices of loops passing through the cylinder. These numbers are discrete invariants and are thus constant in a given chamber. Since the total angle is $(2a+2)\pi$ and there are two ideal triangles in the cylinder, the two angles sum to $2a\pi$. They can be encompassed in an integer unordered bipartition of $a$. Topological indices of the loops determine to which connected component a chamber belongs to.\newline
The shape of the cylinders depends on the choice of a holonomy vector between the two ends of the cylinder. Two such holonomy vectors whose difference is an integer multiple of the period of the closed geodesics of the cylinder represent the same cylinder (Dehn twist). If we choose the periods of the closed geodesics of the cylinder to be equal to $1$, the periods of the holonomy vectors from one end to the other end are complex numbers with positive imaginary part. The difference between two such periods is an integer real number (one for each Dehn twist). Therefore, the chamber is parametrized by $\mathbb{H}/\mathbb{Z}$. Since the action of $GL^{+}(2,\mathbb{R})$ is transitive on bases of $\mathbb{R}^{2}$, it is immediate that two flat surfaces of the same chamber belong to the same $GL^{+}(2,\mathbb{R})$-orbit. A cusp corresponds to an infinitely long holonomy vector.\newline
There is a unique boundary component for a chamber of "cylinder" type whose characteristic bipartition is $k+(a-k)$. It is formed by flat surfaces with degenerate core whose four angles are $\pi,2k\pi,\pi,(a-k)2\pi$. A way to parametrize surfaces of the boundary is to use the real ratio between the lengths of the two saddle connections. Since we do not distinguish between the two saddle connections, we have a unique cusp in the wall of the chamber (corresponding to the case where one of the two saddle connections becomes infinitely longer than the other).
\end{proof}

The description of chambers of "triangle" type needs the introduction of a specific combinatorial structure. We would like to have partitions of an integer $a$ into three parts that would be cyclically ordered. This structure should be sufficiently well-behaved to be acted upon by some linear operators.

\begin{defn}
For any $a \geq 3$, we define $\mathcal{E}_{a} \subset (\mathbb{Z}/a\mathbb{Z})^{3}$ such that $(x,y,z) \in \mathcal{E}_{a}$ if and only if $x,y,z \neq 0$ and $x+y+z=0$. We also define $\mathcal{E}_{a}^{k} = \lbrace (x,y,z \in \mathcal{E}_{a}~\vert~ gcd(x,y,z)=k \rbrace$.\newline
We consider the linear applications $S:(x,y,z)\rightarrow (-x,-y,-z)$ and $T:(x,y,z) \rightarrow (y,z,x)$. The group $Gr(S,T)$ generated by $S$ and $T$ acts on $\mathcal{E}_{a}$ while preserving the stratification $\mathcal{E}_{a} = \bigcup \limits_{k \vert a} \mathcal{E}_{a}^{k}$.\newline
We denote by $\mathcal{B}_{a}^{k} = \mathcal{E}_{a}^{k} / Gr(S,T)$ the set of \textit{cyclic tripartitions} of $a$ with $k$ as greatest common factor.
\end{defn}

In chambers of "triangle" type, the discrete invariant is more complicated. There are three saddle connections and six angular sectors (three that belong to the triangle and three that belong to the domain of the pole). For every external angle, there is an internal angle such that the difference between their magnitudes is an integer multiple of $2\pi$. These three numbers corresponding to the three sides of a cylinder form a \textit{cyclic tripartition} of $a$.

\begin{prop}
In chambers of "triangle" type of $\mathbb{P}\mathcal{H}(a,-a)$, we normalize the area of the ideal triangle to one. Flat surfaces are characterized by two invariants:\newline
(i) A cyclic tripartition of $a$, that is an element of $\mathcal{B}_{a}$. This discrete invariant distinguishes the different chambers of the same type in the stratum.\newline
(ii) The shape of the triangle (angles and side lengths). Consequently, every such chamber is the image of $GL^{+}(2,\mathbb{R})$ under projection.\newline
In the boundary of such chambers, the ideal triangle degenerates and the core becomes a pair of saddle connections. Flat surfaces of the boundary are parametrized by the real ratio between their lengths.\newline
These chambers are topological disks with three boundary components drawn between cusps.\newline
There is a chamber of "triangle" type for every \textit{cyclic tripartition} of $a$.
The chamber corresponding to a tripartition $a=k+k'+k''$ belongs to the connected component where the rotation number of surfaces is $gcd(k,k',k'',a)$.\newline
If $a$ is of the form $3b$, the cyclic tripartition $a=b+b+b$ provides a chamber with an additional symmetry of order three. Such a chamber is a topological disk with a unique boundary component. There is a unique cusp in the boundary and an orbifold point of order three in the interior of the chamber (corresponding to an equilateral triangle).
\end{prop}

\begin{proof}
The geometry of the flat surface is determined by the three angles of the triangle and the three angles of the domain of the pole. For each internal angle, there is an external angle that is formed by the same pair of saddle connections. Their magnitudes differ by a nonzero integer multiple of $2\pi$ (nonzero because the magnitude of an angle of a domain of a pole is at least $\pi$). Besides, this number is the topological index of a loop getting around these two angles. These numbers are discrete invariants and are thus constant in a given chamber.\newline
These three numbers form a tripartition $x+y+z=a$. For practical purposes, we represent them as triplets $(\overline{x},\overline{y},\overline{z})$ of elements of $\mathbb{Z}/a\mathbb{Z}$ such that $\overline{x}+\overline{y}+\overline{z} = 0$. If we consider them up to cyclic permutation and global inversion, every such equivalence class appears from a tripartition of $a$. Indeed, if we have $x+y+z = 2a$, then we have $(a-x)+(a-y)+(a-z)=a$.\newline
Since elements of the discrete invariant are topological degrees of loops, their greatest common divisor decides which connected component the chamber belongs to.\newline
Two nondegenerate triangles are always equivalent under the action of $GL^{+}(2,\mathbb{R})$. Therefore, the interior of a chamber of "triangle" type is an open $GL^{+}(2,\mathbb{R})$-orbit.\newline
In the boundary of the chamber, the triangle degenerates and we get a surface with only two saddle connections. We assume the discrete invariant of the chamber is $(x,y,z)$. We get a flat surface with degenerate core whose four angles are $\pi,2x\pi,\pi+2y\pi,2z\pi$. By permuting cyclically the numbers $x,y,z$, we get three types of flat surfaces with degenerate core. Therefore, there are three boundary components in a chamber. In each of these boundary components we parametrize surfaces using the real ratio between the lengths of the two saddle connections. The chamber is a topological disk bounded by three segments separated by cusps.\newline
In the case $a=3b$, there is a chamber corresponding to the cyclic tripartition $(b,b,b)$. Such a chamber exhibits an exceptional symmetry. The flat surface of the chamber whose triangle is equilateral is an orbifold point of order three and there is a unique boundary component (the topological disk has been quotiented by the group of rotation of order three).
\end{proof}

We end the description of chambers with the case of chambers of "degenerate" type. In such chambers, flat surfaces are characterized by four angles and the length of the two saddle connections forming the core. One of the main differences with the other types of chambers is that they can be formed by several open $GL^{+}(2,\mathbb{R})$-orbits.\newline

For reasons of convenience, we draw a distinction between two classes of chambers of "degenerate" type. If there is a pair of opposed angles with exactly the same magnitude, we refer to the chamber as a chamber of "balanced degenerate" type. Otherwise, we refer to the chamber as a chamber of "unbalanced degenerate" type. The following proposition proves the consistency of the distinction.

\begin{prop}
The difference of magnitude between each pair of opposed angles in flat surfaces that belong to chambers of "degenerate" is a discrete invariant. 
\end{prop}

\begin{proof}
The geometry of the flat surface is determined by the four angles of the domain of the pole and the length of the two saddle connections. Just like in a parallelogram, the difference between the magnitudes of two opposite angles is $2k\pi$ where $k$ is an integer. The angles vary continuously inside the chamber so this number is a discrete invariant of the chamber.\newline
\end{proof}

We begin with a description of the chambers of "balanced degenerate" type. They may feature additional symmetries and include an orbifold point.

\begin{prop}
In strata $\mathbb{P}\mathcal{H}(a,-a)$, chambers of "balanced degenerate" type are distinguished by a bipartition $k+k'$ of $a$. For any such bipartition where $k,k' \geq 1$, there is a unique chamber of "balanced degenerate" type. The chamber belongs to the connected components where the rotation number of the flat surfaces is $gcd(k,a)$.\newline
If $a$ is an even number, the chamber associated to the bipartition $\dfrac{a}{2}+\dfrac{a}{2}$ has an exceptional symmetry. The chamber is a topological disk with a unique cusp in the boundary and unique boundary component. Besides, there is an orbifold point of order two inside the chamber (corresponding to a surface where the two saddle connections have same length and the four angles are equal).\newline
Any other chamber of "balanced degenerate" is a topological disk with two boundary components separated by two cusps.\newline
In the boundary of such chambers, the saddle connections share the same direction and at least one angle has magnitude equal to $\pi$. Such flat surfaces are parametrized by the real ratio between the lengths of the saddle connections.
\end{prop}

\begin{proof}
In a chamber of "degenerate" type, there are four corners in the boundary of the core (see Figure 3). They are cyclically ordered around the unique conical singularity. There is a pair of opposite angles with the same magnitude. Besides, the sum of the magnitudes of consecutive angles is an invariant of such chambers because it is given by the topological index of a loop. Therefore, the four angles are of the form $\alpha,\pi-\alpha+2k\pi,\alpha,\pi-\alpha+2k'\pi$ where $\alpha > \pi$, $k$ and $k'$ are topological indices of loops forming a homology basis of the torus and $k+k'=a$. Bipartition $k+k'$ of $a$ is a discrete invariant of the chamber and the chamber belongs to the connected component where flat surfaces have a rotation number of $gcd(k,a)$.
\newline
A flat surface surface that belongs to such a chamber is characterized by the lengths of its two saddle connections and its angles $\alpha,\pi-\alpha+2k\pi,\alpha,\pi-\alpha+2k'\pi$. We consider two such surfaces $X_{0}$ and $Y$. For any pair of such surfaces $X_{0}$ and $Y$, we can consider the path $X_{t}$ formed by surfaces whose saddle connections have the same length as that of $X_{0}$ and such that $\alpha_{t}=\alpha+t.(\beta-\alpha)$ where two angles of $Y$ have magnitude $\beta$. The angles stay clearly above $\pi$ all along the path so for every $t \in [0;1]$, $X_{t}$ belongs to the same chamber. Since they have the same angles, $X_{1}$ and $Y$ belong to the same $GL^{+}(2,\mathbb{R})$-orbit and thus to the same chamber. Therefore, two flat surfaces with the same discrete invariant belong to the same invariant. This construction also ensures that the chamber is contractible.\newline
If $a$ is an even number, we can consider the chamber whose discrete invariant is $\dfrac{a}{2}+\dfrac{a}{2}$. In the chamber, opposite angles are equal and the flat surfaces of the chamber have a symmetry of order two. Besides, there is flat surface whose four angles are equal and whose saddle connections have the same length. This surface is exceptionaly symmetric (order four) and corresponds to an orbifold point of order two. The boundary of the chamber is formed by surfaces whose angles are $\pi,a\pi,\pi,a\pi$. Therefore, there is a unique boundary component.\newline
For any other discrete invariant, there are two boundary components. Indeed, for a flat surface in the boundary, there is an angle equal to $\pi$. Depending on the pair of opposite angles this angle belongs to, the flat surface belongs to a boundary component or the other. These pairs are combinatorially different because by hypothesis there is one where the angles are equal and another where they are different. The case where the two pairs displays equal angles has been dealt with.
\end{proof}

Chambers of "unbalanced degenerate" type are more common and do not feature additional symmetries.

\begin{prop}
In strata $\mathbb{P}\mathcal{H}(a,-a)$, chambers of "unbalanced" type are distinguished by a combinatorial invariant that is an ordered tripartition $x+y+z$ of $a$. For every such triplet, there is another triplet that corresponds to the same chamber. If $z>x$, then $(x+y)+(z-x)+x$ represents the same chamber as $x+y+z$. If $x>z$, then it is $z+(x-z)+(y+z)$ that represents the same chamber. For every such pair of ordered tripartitions of $a$, there is a unique chamber of "unbalanced" type. The chamber belongs to the connected component where the rotation number of the flat surfaces is $gcd(x,y,z)$.\newline
These chambers are topological disks with two boundary components separated by two cusps. In the boundary of such chambers, the saddle connections share the same direction and at least one angle has magnitude equal to $\pi$. Such flat surfaces are parametrized by the real ratio between the lengths of the saddle connections.
\end{prop}

\begin{proof}
The sum of the magnitudes of consecutive angles and the difference of the magnitudes of opposite angles are discrete invariants of such chambers. We assume the magnitude of an angle is equal to $\pi$ and define the other angles (in the cyclic order) as $2x\pi$, $(2y+1)\pi$ and $2z\pi$. Therefore, $x+y+z$ is a tripartition of $a$. Since the difference of magnitudes between opposite angles encompasses the whole discrete invariant (defined by topological indices of loops), another triplet may represent the same discrete invariant. If $z>x$, then the other triplet is $(x+y),(z-x),x$. If $x>z$, then it is $z,x-z,z+y$. The idea is than we reduce the angles of a pair of opposite angles to increase the angles of the other pair. The two tripartitions describe the two ends of this process (and therefore the two boundary components of the chamber). Since the sums of magnitudes of consecutive angles is given by the topological index of loops of a basis of the homology of the torus, the chamber belongs to the connected component where the rotation number of the flat surfaces is $gcd(x,y,z)$.\newline
Just like in the proof of Proposition 4.6, existence of a linear path relating the angles of any two flat surfaces with the same discrete invariants implies that there is a unique (and contractible) chamber for a given discrete invariant.\newline
If the difference of magnitude of the two pairs of opposite angles are not the same, then it is clear that there are two different boundary components in the chamber. The other case is when the differences are the same. In this case, there are also two different boundary components and they are distinguished by the orientation. For example, if $a=6$, the chamber corresponding to $3+1+2$ (which corresponds to the same chamber as $2+1+3$) has two boundary components where angles are respectively $\pi,6\pi,3\pi,4\pi$ and $\pi,4\pi,3\pi,6\pi$. Such chambers thus are topological disks bounded by two segments related by two cusps.
\end{proof}

\subsection{Adjacency relations between chambers}

Every chamber of every type is characterized by a discrete combinatorial invariant. In this subsection, we provide combinatorial criteria to decide which pair of chambers share a boundary component.\newline

\begin{prop}
For a connected component of a stratum $\mathbb{P}\mathcal{H}(a,-a)$, every boundary component of the walls-and-chambers structure separates a chamber of "degenerate" type and a chamber of one of the two other types.\newline
The chamber of "balanced degenerate" type corresponding to the bipartition $k+k'=a$ is connected with the chamber of "cylinder" type corresponding to the bipartition $k+k'=a$. If there is another boundary component, it is shared with the chamber of "triangle type" corresponding to the tripartition $k+k+(k'-k)$ (with $k'>k$).\newline
The chamber of "unbalanced degenerate" type corresponding to two tripartitions of $a$ (see Proposition 4.7) connects the two chambers of "triangle" type corresponding to these tripartitions.
\end{prop}

\begin{proof}
In flat surfaces that belong to the boundary of chambers, there is at least an angle whose magnitude is $\pi$ between two distinct saddle connections. These two saddle connections generate the homology of the torus. Since they share the same direction, such surfaces have a core whose area is zero (and thus is degenerate). The angles of these surfaces verify the relations (modulo $2\pi$) of parallelograms. Therefore, the only angle whose magnitude could fall below $\pi$ as we leave the boundary component is the opposite angle of the angle of magnitude $\pi$. However, these angles grow and decrease together. Therefore, there is always a chamber on a side of the boundary component where the four angles stay above $\pi$. This chamber is of "degenerate" type.\newline
In the boundary of a chamber of "cylinder" type corresponding to the bipartition $k+k'$, the four angles of the flat surfaces are $\pi,2k\pi,\pi,2k'\pi$. By transporting the magnitudes of the angles for a pair of opposite angles to the other, we get another boundary component where the four angles are $2k\pi,\pi,2k\pi,(2k'-2k+1)\pi$. Such a boundary component is shared with the chamber of "triangle" type corresponding to the tripartition $k+k+(k'-k)$. Indeed, there is a unique angle whose magnitude is $\pi$. Therefore, a perturbation can only produce a unique ideal triangle.\newline
In the boundary of a chamber of "unbalanced" type, the four angles of the flat surfaces include only one angle of magnitude $\pi$. Therefore, the two chambers that share a boundary with a chamber of "unbalanced" type are chambers of "triangle" type.
The two boundaries correspond to each pair of opposite angles. For each boundary component, there is a unique angle equal to $\pi$ and the repartition of a total angle of $(2a+1)\pi$ provide a tripartition characteristic of a chamber of "triangle" type. The two triplets corresponding to the chamber of "unbalanced" type correspond to a choice of a pair of opposite angles. Therefore, a chamber of "unbalanced" type is connected with the two chambers of "triangle" type having cyclic tripartitions corresponding to the two triplets.
\end{proof}

Now we have a full description of chambers and their relations, we are able to prove that the topology and walls-and-chambers structure of connected components of the strata only depend on the ratio between the degree $a$ of the stratum and the rotation number of the connected component.

\begin{proof}[Proof of Theorem 2.3]
We consider the stratum $\mathbb{P}\mathcal{H}(d.a,-d.a)$ and its connected component $\mathcal{C}_{d}^{a.d}$ where the rotation number of the surfaces is $d$ with $a\geq2$ and $d\geq1$. It is already proved in Theorem 2.2 that $\mathcal{C}_{d}^{a.d}$ is biholomorphic to modular curve $X_{1}\left( \dfrac{a}{d}\right)$ and therefore to connected component $\mathcal{C}_{1}^{a}$.\newline
Since the rotation number of the connected component a chamber belongs to is a greatest common divisor of numbers forming the combinatorial invariant of the chamber (cyclic tripartitions or bipartitions), invariants of the chambers of $\mathcal{C}_{d}^{a.d}$ are of the form $d.A$ where $A$ is a cyclic tripartition or a bipartition. The condition underwhich such a combinatorial data corresponds to a chamber is also homogeneous. Therefore, $A$ is the combinatorial invariant of a chamber of stratum $\mathbb{P}\mathcal{H}(a,-a)$. Such a chamber belongs to the connected component $\mathcal{C}_{1}^{a}$ where the rotation number is $1$.\newline
Similarly, the combinatorial relations that describe adjacentness relations between chambers are also homogenous of degree one. Therefore, the graph of adjacentness relations between chambers is the same in $\mathcal{C}_{d}^{a}$ and $\mathcal{C}_{1}^{a}$.\newline
Chambers are topological disks so there is a very natural CW-complex structure associated with the walls-and-chambers structure of a connected component. Since there is a correspondance between chambers of the two components and isomorphism between their graph of adjacentness, $\mathcal{C}_{d}^{d.a}$ and $\mathcal{C}_{1}^{a}$ are isomorphic as CW-complexes. This ends the proof.
\end{proof}

\subsection{Description of the exceptional components isomorphic to $X_{1}(2)$, $X_{1}(3)$ and $X_{1}(4)$}

In projectivized strata of the form $2b$ or $3b$, the connected component corresponding the rotation number $b$ displays special features. Indeed, they have orbifold special points that no not appear for example in principal components of strata $\mathbb{P}\mathcal{H}(a,-a)$ when $a\geq 4$. We provide a complete topological description of these exceptional components in order to simplify the further investigation. Since these components are isomorphic to either $\mathbb{P}\mathcal{H}(2,-2)$ or $\mathbb{P}\mathcal{H}(3,-3)$, we describe them directly. \newline
The principal connected component of $\mathbb{P}\mathcal{H}(4,-4)$ (which is isomorphic to $X_{1}(4)$) also displays a small deviation from the general formula we prove in Theorem 2.4.\newline

\begin{prop} $\mathbb{P}\mathcal{H}(2,-2)$ is a topological sphere with two punctures and an orbifold point of order two. The discriminant is formed by a unique arc connecting one puncture with itself and cutting out two chambers. These chambers are a chamber of "cylinder" type (with a puncture in its interior) and a chamber of "degenerate" type (with the orbifold point in its interior).
\end{prop}

\begin{proof}
Following Proposition 4.4, there is no chamber of "triangle" type because there is no tripartition of $2$. For the same reason, there is no chamber of "unbalanced degenerate" type (Proposition 4.7). Since $1+1$ is the unique partition of $2$, there is a unique chamber of "cylinder" type and a unique chamber of "balanced degenerate" type (Propositions 4.2 and 4.6). These two chambers both have a unique boundary component. Thus, there are two chambers in $\mathbb{P}\mathcal{H}(2,-2)$. These two chambers are separated by an arc connecting a puncture to itself. The chamber of "cylinder" type is a topological disk with one puncture inside (corresponding to an infinitely stretched cylinder). Besides, there is an orbifold point in the chamber of "degenerate" type (Proposition 4.6).\newline
To sum up, if we forget the puncture inside the chamber of "cylinder" type and the orbifold point, the projectivized chamber can be decomposed into two faces, one edge and one vertex. Therefore, the Euler characteristic of the surface is $2$ and the genus of the surface is zero.
\end{proof}

\begin{prop} $\mathbb{P}\mathcal{H}(3,-3)$ is a topological sphere with two punctures and an orbifold point of order three. The discriminant is formed by two arcs connecting one puncture to itself, cutting out three chambers in a linear order. These chambers are (following the order) a chamber of "cylinder" type, a chamber of "degenerate" type and a chamber of "triangle" type. The chamber of "cylinder" type contains the other puncture (corresponding to an infinitely stretched cylinder) whereas the chamber of "triangle" type contains the orbifold point (corresponding to an equilateral triangle).
\end{prop}

\begin{proof}
There is a unique chamber of "cylinder" type and its characteristic partition is $1+2$ (Proposition 4.2). There is a unique chamber of "triangle" type and its cyclic tripartition is $1+1+1$ (Proposition 4.4). A chamber of "triangle" type can have three boundary components (one for each way for the triangle to degenerate). However, this chamber has a cyclic symmetry of order three. Therefore, it has a unique boundary component and an orbifold point of order three (corresponding to the equilateral triangle). Since chambers of "degenerate" type are adjacent to chambers of "cylinder" and "triangle" type only, these two chambers are connected by a unique chamber of "degenerate" type.\newline
The discriminant is formed by two arcs whose ends correspond to limit cases where some saddle connections shrink. Since the boundaries of the chambers of "cylinder" and "triangle" type are closed segments, there is only one such puncture. If we forget the puncture inside the chamber of "cylinder" type and the orbifold point inside the chamber of "triangle" type, the discriminant provides a decomposition of the stratum into three faces, two edges and one vertex. Therefore, the Euler characteristic of the surface is $2$ and the genus of the surface is zero.
\end{proof}

\begin{prop} There are two connected components in $\mathbb{P}\mathcal{H}(4,-4)$. The component where the rotation number is $2$ is isomorphic to $\mathbb{P}\mathcal{H}(2,-2)$. On the contrary, the principal connected component (where the rotation number is $1$) is a topological sphere with three punctures. There are four chambers : two that belong to the "degenerate" type and one of the other other types, see $\mathcal{G}_{4}$ in Figure 1 for its graph of adjacentness.
\end{prop}

\begin{proof}
Following Theorem 2.2, there are two connected components in $\mathbb{P}\mathcal{H}(4,-4)$ that correspond respectively to the rotation numbers $1$ and $2$. Following Theorem 2.3, connected component $\mathcal{C}^{4}_{2}$ is isomorphic to $\mathbb{P}\mathcal{H}(2,-2)$. In particular, their walls-and-chambers structures are isomorphic.\newline
Next, we focus on the connected component $\mathcal{C}^{4}_{1}$. Following Proposition 4.2, there is a unique chamber of "cylinder" type because there is a unique bipartition $1+3=4$ that satisfies the coprimality condition. Likewise, following Proposition 4.4, there is a unique tripartition $1+1+2=4$ so there is a unique chamber of "triangle" type. Therefore, there is a chamber of "degenerate" type that connects the two latter chambers. There is also another chamber of "degenerate" type that connects the two remaining boundaries of the chamber of "triangle" type. Indeed, two chambers of "degenerate" type cannot be directly adjacent. We get $\mathcal{G}_{4}$ in Figure 1 as graph of adjacentness of the chambers.\newline
Since there is a unique puncture in the boundary of the chamber of "cylinder" type, it is also the same puncture in the boundary of the next chamber of "degenerate" type and one of the two vertices in the boundary of the chamber of "triangle" type. However, the remaining vertex of the boundary of the chamber of "triangle" type is not identified with the others. We add the puncture that belong to the interior of the chamber of "cylinder" type and get three distinct punctures. There are $4$ chambers, $2$ punctures (one of the three is not used in the cell decomposition) and $4$ arc components of the discriminant. Therefore, the Euler characteristic is $2$ and the component is a topological sphere.
\end{proof}

\subsection{Graph of chambers in connected components of strata}

Just like in the previous subsection, we consider only principal connected components of strata. Other components are isomorphic to such components for a different stratum (see Theorem 2.3).\newline
Since adjacency relations have been reduced to purely combinatorial issues, there is a systematic process to draw the graph of adjacency of chambers of a connected component of $\mathbb{P}\mathcal{H}(a,-a)$.\newline

Most of the form of the graph is provided by the relations between chambers of "triangle" type. Hopefully, this graph is easy to handle.

\begin{defn}
We denote by $\mathcal{T}_{a}$ the graph such that:\newline
- the vertices are the chambers of "triangle" type of the principal connected component of $\mathbb{P}\mathcal{H}(a,-a)$.\newline
- the edges are the chambers of "unbalanced degenerate" type that relate two chambers of "triangle" type.
\end{defn}

The graphs $\mathcal{T}_{a}$ are very regular, see Figure 5. They are Schreier coset graphs. In other words, they are given by the quotient of a finite group by a subgroup.\newline
It should be recalled that $\mathcal{E}_{a}^{1}$ is the set of triplets $(x,y,z)$ of elements of $\mathbb{Z}/a\mathbb{Z}$ such that we have:\newline
(i) $x,y,z \neq 0$,\newline
(ii) $x+y+z=0$,\newline
(iii) $gcd(x,y,z)=1$.\newline
In addition of the operators $S:(x,y,z)\rightarrow (-x,-y,-z)$ and $T:(x,y,z) \rightarrow (y,z,x)$ (see Definition 4.3), we also define linear application $U:(x,y,z) \rightarrow (x+z,y,x-z)$.

\begin{figure}
\includegraphics[scale=0.3]{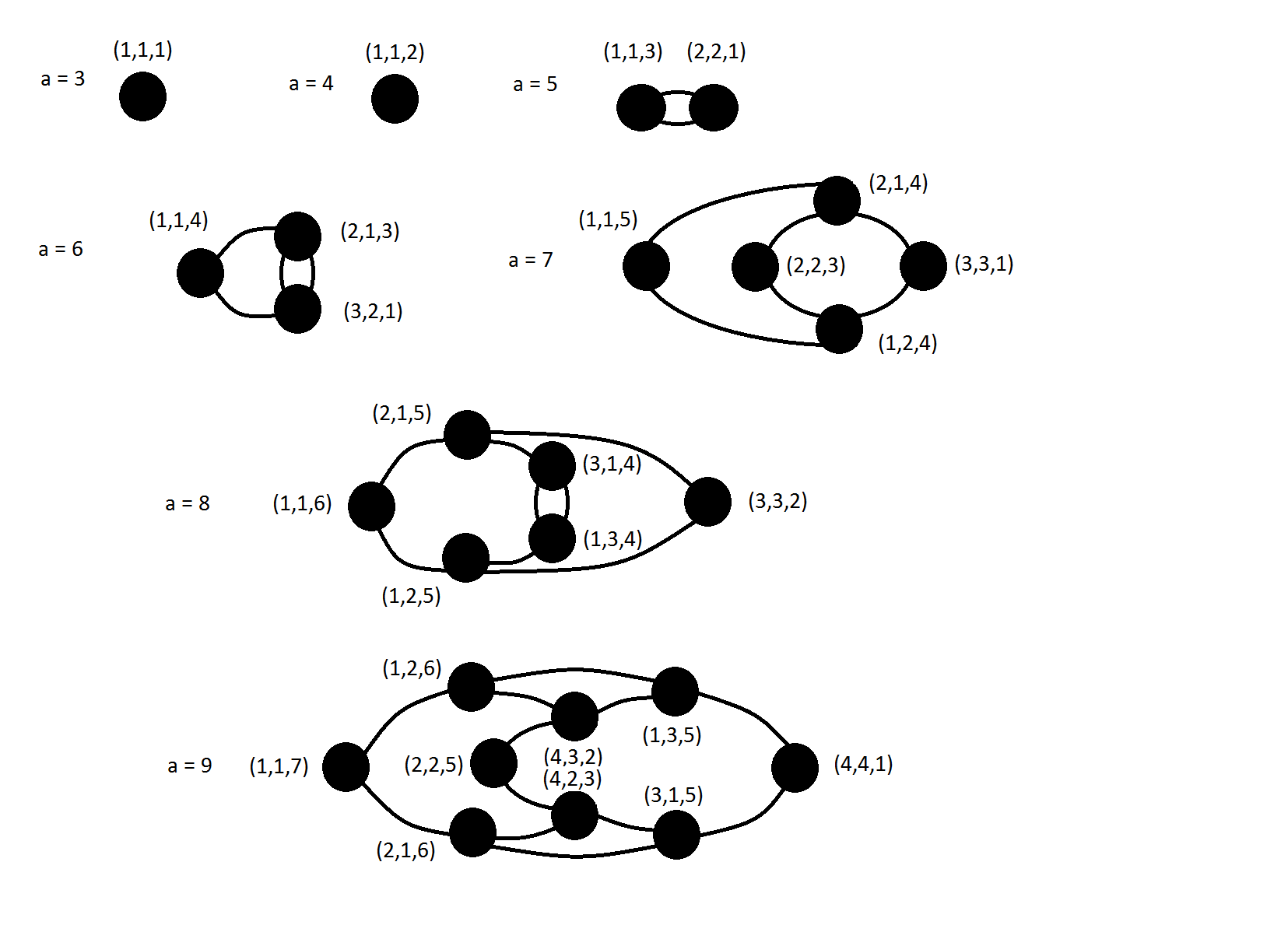}
\caption{The first elements of the family of graphs $\mathcal{T}_{a}$.}
\end{figure}

\begin{prop}
For any $a \geq 3$, graph $\mathcal{T}_{a}$ is isomorphic to the graph defined by:\newline
(i) $\mathcal{B}_{a}^{1} = \mathcal{E}_{a}^{1} / Gr(S,T)$ as set of vertices,\newline
(ii) an edge between two classes $\overline{\alpha}$ and $\overline{\beta}$ of $\mathcal{B}_{a}^{1}$ such that there are representatives $\alpha$ and $\beta$ in $\mathcal{E}_{a}^{1}$ with $\beta=U(\alpha)$.\newline
$Gr(S,T,U)$ acts transitively on $\mathcal{E}_{a}^{1}$. Therefore, $\mathcal{T}_{a}$ is a Schreier coset graph.
\end{prop}

\begin{proof}
It has been proven in Proposition 4.4 that chambers of "triangle" type in the principal connected component are in correspondance with elements of $\mathcal{B}_{a}^{1} = \mathcal{E}_{a}^{1} / Gr(S,T)$. These chambers are connected by a chamber of "unbalanced degenerate" type if and only if the invariant of one chamber is represented by $(x,y,z)$ and the other by $((x+y),y,(z-x))$ or $(z,(x-z),(y+z)$ depending on the sign of $x-z$. In other words, a representative (up to cyclic permutation and inversion) of the discrete invariant of each chamber should be related by the linear map $U$.\newline
The action of $Gr(S,T,U)$ acts transitively on $\mathcal{E}_{a}^{1}$ because otherwise, there would be several connected components with rotation number equal to one (the classification of connected components has been established by other means in \cite{Bo}).
\end{proof}

While the graph defined by the chambers of "triangle" type is very regular, the actual graph of chambers of the connected components is a little more complicated, see Figure 6. For the picture of the graphs of chambers for small values of $a$, see Figure 1.

\begin{figure}
\includegraphics[scale=0.3]{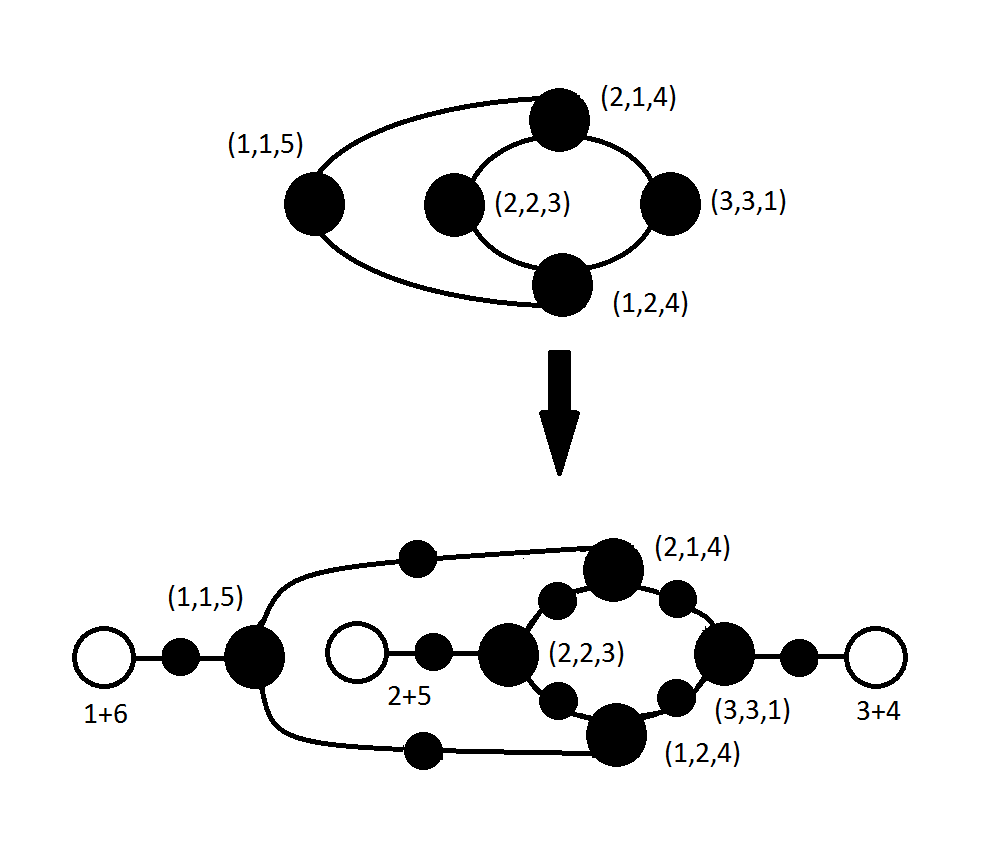}
\caption{$\mathcal{T}_{7}$ and $\mathcal{G}_{7}$. Small black disks correspond to chambers of "degenerate type and the white circle corresponds to a chamber of "cylinder" type.}
\end{figure}

\begin{thm}
The graph of chambers $\mathcal{G}_{a}$ of the principal connected component of $\mathbb{P}\mathcal{H}(a,-a)$ is given by the two following operations:\newline
(i) adding a leaf (a vertex and an edge) on every vertex of $\mathcal{T}_{a}$ of the form $(a-2\theta,\theta,\theta$ (corresponding to chambers of "cylinder" type).\newline
(ii) inserting a vertex (corresponding to chambers of "degenerate" type) on every edge to get $\mathcal{G}_{a}$.
\end{thm}

\begin{proof}
Following Proposition 4.8, chambers of "cylinder" type corresponding to bipartitions $k+k'$ are connected by a chamber of "balanced degenerate" type to a chamber of "triangle" type whose discrete invariant is $k+k+(k'-k)$. Adding a leaf on every vertex corresponding to tripartitions of the form $(a-2\theta)+\theta+\theta)$ provides a graph where chambers of "cylinder" and "triangle" type are represented. Inserting a vertex on every edge provides a graph where the vertices correspond to chambers and edges correspond to boundary components. 
\end{proof}

\section{Counting cells in the connected components}
 
In the following, we give formulas to provide combinatorial information on strata. We denote by $S_{a}$ the number of connected components of the discriminant in $\mathbb{P}\mathcal{H}(a,-a)$. We also denote by $F_{a}$ the number of connected components of the discriminant in the connected component of $\mathbb{P}\mathcal{H}(a,-a)$ whose rotation number is one. Since $\mathcal{H}(1,-1)$ is empty, we set $F_{1}=S_{1}=0$\newline

\begin{prop} For any $a\geq2$, we have
$S_{a}=\lfloor \frac{a}{2} \rfloor + \dfrac{(a-1)(a-2)}{2}$ and
$F_{a}=\sum \limits_{d \vert a} \mu(a/d).S_{d}$ where $\mu$ is the classical Möbius function.
\end{prop}

In particular, when $a$ is equal to a prime number $p\geq5$, there are $\frac{(p-1)^{2}}{2}$ connected components in the discriminant of the stratum (which is connected).

\begin{proof}
Flat surfaces that belong to the discriminant have a degenerate core and one of the four angles of the domain of the pole is $\pi$. Connected components of the discriminant are parametrized by the lengths of the saddle connections and correspond to choices of angles. In some surfaces, there is a pair of opposite angles whose magnitude is equal to $\pi$. There is a connected component for every choice of angles that are multiple of $2\pi$ for the other pair of opposite angles. Such connected components are in bijection with bipartitions of $a$ into two integers. There are $\lfloor\dfrac{a}{2}\rfloor$ of them. For connected components where the opposite angle is not equal to $\pi$, there is a connected component for every ordered partition of $a$ into three parts. Therefore, there are $\sum \limits_{k=1}^{a-2} (a-k-1)=(a-1)(a-2)-\dfrac{(a-1)(a-2)}{2}=\dfrac{(a-1)(a-2)}{2} $ of them.\newline
By definition, $S_{a}=\sum \limits_{d \vert a} F_{d}$. By the Möbius inversion formula, we get $F_{a}=\sum \limits_{d \vert a} \mu(a/d).S_{d}$.
\end{proof}

We compute the number of chambers of each type in the principal connected component:\newline
(i) $CT_{a}$ is the number of chambers of "triangle" type.\newline
(ii) $DT_{a}$ is the number of chambers of "degenerated" type.\newline
(iii) $CC_{a}$ is the number of chambers of "cylinder" type.\newline
We denote by $C_{a}$ the number of chambers in the principal connected component. It should be recalled that $\phi$ is Euler's totient function

\begin{prop} For any $a\geq4$, we have:\newline
(i) $DT_{a}=\dfrac{1}{2}F_{a}$.\newline 
(ii) $CC_{a}=\frac{1}{2}\phi(a)$.\newline
(iii) $CT_{a}=\dfrac{1}{3}F_{a}-\frac{1}{6}\phi(a)$.\newline
(iv) $C_{a}=\dfrac{5}{6}F_{a}+\frac{1}{3}\phi(a)$.
\end{prop}

In particular, when $a$ is equal to a prime number $p\geq5$, there are $\dfrac{(p-1)(5p-1)}{12}$ chambers in the stratum (which is connected).
\begin{proof}
Since we excluded the case $a=2$, every chamber of "degenerate" type has exactly two boundary components. Therefore, we have $DT_{a}=\dfrac{1}{2}F_{a}$.\newline
In $\mathbb{P}\mathcal{H}(a,-a)$, there is a chamber of "cylinder" type for every bipartition of $a$. Since the rotation number of flat surfaces where the bipartition is $k+(a-k)$ is $gcd(k,a)$ and Euler's totient function counts coprime numbers, there are exactly $\frac{1}{2}\phi(a)$ chambers of "cylinder" type.\newline
Finally, every chamber of "triangle" type has three boundary components (since we excluded the case $a=3$) so $3CT_{a}=F_{a}-CC_{a}$. Therefore, we have $CT_{a}=\dfrac{1}{3}F_{a}-\frac{1}{6}\phi(a)$. Finally, we have $C_{a}=CT_{a}+CC_{a}+DT_{a}$.
This ends the proof.\newline
\end{proof}

Having the number of chambers and the number of connected components of the discriminant, the computation of the number of punctures leads directly to the Euler characteristic of the complex curve and the computation of its genus.

\begin{proof}[Proof of Theorem 2.4]
The cases $2 \leq a \leq 4$ are already settled in Propositions 4.9 to 4.11 so we assume $a \geq 5$.\newline
We first draw a distinction among the cusps of the principal connected component of $\mathbb{P}\mathcal{H}(a,-a)$. Some belong to the interior of chambers of "cylinder" type and the others separate the connected components of the discriminant. They belong to the boundary of the chambers. There is exactly one cusp in each of the $\frac{1}{2}\phi(a)$ chambers of "cylinder" type. We put them aside as of now.\newline
For every cusp that separates the connected components of the discriminant, there is a small closed loop around it that passes through several chambers. This loop defines a closed path in the graph $\mathcal{T}_{a}$ of chambers of "triangle" type. Indeed, the cusp that belongs to the boundary of a chamber of "cylinder" type of invariant $k+k'$ also belongs to the boundary of the chamber of "triangle" type of invariant $k+k+(k'-k)$. Therefore, we focus on the chambers of "triangle" type.\newline
For every cusp in the boundary of a chamber of "triangle" type, there are two vertices of the triangle that merge and one that remains alone. For a chamber of invariant $(x,y,z)$, we get a limit-surface of genus zero with a conical singularity of order $y-1$ and one other of degree $x+z-1$. A loop around a cusp means a path among chambers of "triangle" type where one of the three degrees of the invariant is constant. Therefore, this loop is a finite orbit in $\mathcal{E}_{a}^{1}$ under the action of $U$.\newline
In each of these orbits, the middle element $y$ of the triplet remains constant. Let $d(y)=gcd(y,a)$. The number $d(y)$ may be any divisor of $a$ different from $a$. For every such divisor, we compute the number of loops. That way, we will be able to compute the total number of cusps.\newline
There are $\phi(a/d)$ numbers $y$ such that $d(y)=d$. Besides, the length of the orbit under $U$ is $a/d$. The different orbits sweep the triplets whose $gcd$ is one. Therefore, there is a total of $\phi(d)$ different orbits for such an $y$. Since triplets are considered up to global inversion, we have to divide the total numbers of orbits by a factor $2$ to get the total number of cusps that belong to the discriminant. This procedure does not work when $a=4$ and therefore this case had to be settled separately.\newline
Finally, the number of cusps that belong to the discriminant is $-\frac{1}{2}\phi(a)+\dfrac{1}{2}\sum \limits_{d \vert a} \phi(d)\phi(a/d)$. Consequently, there are a total of $l=\dfrac{1}{2}\sum \limits_{d \vert a} \phi(d)\phi(a/d)$ punctures in the complex curve.\newline
We use the walls-and-chambers structure of the connected component to compute its genus. The chambers are the cells of dimension two, the connected components of the discriminant are the cells of dimension one and the cusps that separates the connected components of the discriminant are the vertices of a polyhedral decomposition.\newline
The Euler characteristic of this polyhedral decomposition is:\newline
$$\chi=C_{a}-F_{a}+\dfrac{1}{2}\sum \limits_{d \vert a} \phi(d)\phi(a/d)-\frac{1}{2}\phi(a)$$.\newline
Therefore, we have $\chi= \dfrac{1}{2}\sum \limits_{d \vert a} \phi(d)\phi(a/d)
-\dfrac{1}{6}F_{a}-\frac{1}{6}\phi(a)$.\newline
Since $\chi=2-2g$ for an orientable surface, we get $g=1+\dfrac{1}{12}(F_{a}+\phi(a))-\dfrac{1}{4}\sum \limits_{d \vert a} \phi(d)\phi(a/d)$.\newline
Following Proposition 5.1, we have $F_{a}=\sum \limits_{d \vert a} \mu(a/d).S_{d}$ where $\mu$ is the classical Möbius function and $S_{a}=\lfloor \frac{a}{2} \rfloor + \dfrac{(a-1)(a-2)}{2}$.\newline
So we have $
F_{a}+\phi(a)=
\sum \limits_{d \vert a} \mu(a/d).\lfloor \frac{d}{2} \rfloor
+ \sum \limits_{d \vert a} \mu(a/d).\dfrac{(d-1)(d-2)}{2}
+\phi(a)$.\newline
It is well known that $\phi(a)=
 \sum \limits_{d \vert a} \mu(a/d).d$. Therefore, we have:\newline
$$F_{a}+\phi(a)=
\sum \limits_{d \vert a} \mu(a/d).\left( \lfloor \frac{d}{2} \rfloor
 + \dfrac{(d-1)(d-2)}{2} + d \right).$$
We define function parity $\rho$ such that $\rho(d)=0$ if $d$ is even and $\rho(d)=1$ if $d$ is odd.\newline
Thus, we get:\newline
$$F_{a}+\phi(a)=
\sum \limits_{d \vert a} \mu(a/d).\dfrac{d^{2}}{2}+\sum \limits_{d \vert a} \mu(a/d).(1-\rho(d))$$
Finally, we have $F_{a}+\phi(a)= \sum \limits_{d \vert a} \mu(a/d).\dfrac{d^{2}}{2}=\dfrac{1}{2}\prod \limits_{p \vert a}(1-\dfrac{1}{p^{2}})$.\newline
Consequently, we have $g=1+\dfrac{a^{2}}{24}\prod \limits_{p \vert a}(1-\dfrac{1}{p^{2}})-\dfrac{1}{4}\sum \limits_{d \vert a} \phi(d)\phi(a/d)$.
\end{proof}
\nopagebreak
\vskip.5cm

\begin{thebibliography}{10}
  
\bibitem{Bo}
C. Boissy.
\newblock {\em Connected components of the strata of the moduli space of meromorphic differentials}.
\newblock {Commentarii Mathematici Helvitici}, Volume 90, Issue 2, 255-286, 2015.
  
\bibitem{BS}
T. Bridgeland, I. Smith.
\newblock {\em Quadratic differentials as stability conditions}.
\newblock {Publications de l'IHES}, Volume 121, Issue 1, 155-278,
  2015.
  
\bibitem{DS}
F. Diamond, J. Shurman.
\newblock {\em A first course in modular forms}.
\newblock {Springer}, 2005.
  
\bibitem{LM}
E. Looijenga, G. Mondello.
\newblock {\em The fine structure of the moduli space of abelian differentials in genus 3}.
\newblock {Geometriae Dedicata}, Volume 169, 120-128,
2014.
  
\bibitem{HKK}
F. Haiden, L. Katzarkov, M. Kontsevich.
\newblock {\em Flat surfaces and stability structures}.
\newblock {Preprint, arXiv:1409.8611},
  2015.

\bibitem{KZ}
M. Kontsevich, A. Zorich.
\newblock {\em Lyapunov exponents and Hodge theory}.
\newblock {Conference to the memory of C. Itzykson, Saclay}
  1997.
  
\bibitem{Ta}
G. Tahar.
\newblock {\em Counting saddle connections in flat surfaces with poles of higher order}.
\newblock {Preprint, arXiv:1606.03705},
  2016.

\bibitem{Zo}
A. Zorich.
\newblock {\em Flat Surfaces}.
\newblock {Frontiers in Physics, Number Theory and Geometry, 439-586},
  2006.


\end{thebibliography}
\end{document}